\documentclass[final]{nyjm} 

\usepackage{slashed}
\usepackage[all]{xy}
\usepackage{amsmath}
\usepackage{amssymb}
\usepackage{amsthm}
\usepackage{graphicx}

\newtheorem{theorem}{Theorem}[section]
\newtheorem{lemma}[theorem]{Lemma}
\newtheorem{prop}[theorem]{Proposition}

\theoremstyle{definition}
\newtheorem{defn}[theorem]{Definition}
\newtheorem{rmk}[theorem]{Remark}
\newtheorem{constr}[theorem]{Construction}
\newtheorem{example}[theorem]{Example}

\newcommand{\cinf}[1]{C^{\infty}(#1)}
\newcommand{\bbR}{\mathbb{R}}
\newcommand{\bbC}{\mathbb{C}}
\newcommand{\bbA}{\mathbb{A}}
\newcommand{\roo}{\mathbb{R}^{1|1}}
\newcommand{\Man}{\mathrm{Man}}
\newcommand{\SMan}{\mathrm{SMan}}
\newcommand{\susyebord}{1|1\textrm{-}\mathrm{EBord}}

\newcommand{\susyeft}{1|1\textrm{-}\mathrm{EFT}}

\newcommand{\diracD}{\slashed{D}}
\newcommand{\dual}[1]{{#1}^{\vee}}
\newcommand{\nqvect}{\mathcal{N}\mathrm{QVect}}

\newcommand{\target}{\textrm{Mod}}
\newcommand{\ENDO}[1]{\mathrm{End}(#1)}
\newcommand{\AUT}[1]{\mathrm{Aut}(#1)}

\title{Supersymmetric Euclidean field theories and K-theory}

\author{Peter Ulrickson}
\address[ulrickson@cua.edu]{116 Aquinas Hall, The Catholic University of America, Washington, DC 20064}
\begin{document}

\maketitle

\begin{abstract}
We construct spaces of 1-dimensional supersymmetric Euclidean field theories and show that they represent real or complex K-theory. A noteworthy feature of our bordism category is that the identity bordism of a point is connected to intervals of positive length.
\end{abstract}


\section{Introduction}

\noindent
The purpose of this paper is to prove the following theorem.

\begin{theorem} \label{mainTheorem}
  There are spaces $\susyeft_\bbC$ and $\susyeft_\bbR$ of one-dimensional, oriented, supersymmetric Euclidean field theories having the homotopy type of $BU \times \mathbb{Z}$ and $BO \times \mathbb{Z}$, respectively.
\end{theorem}

\subsection{Background}
In the Stolz-Teichner program for functorial quantum field theories and cohomology, a number of spaces of $1|1$-Euclidean field theories have been constructed. In each case one obtains a space of operators known to represent K-theory. In \cite{StolzTeichnerelliptic}, field theories are identified with super semigroups of self-adjoint Clifford linear Hilbert-Schmidt operators on a Hilbert space. In \cite{Markert} it is shown that one can obtain a model of the connective ko-theory spectrum by proceeding in a similar manner but considering right and left linearity with respect to two Clifford algebra actions. The paper \cite{HST} shows that a space of $1|1$-Euclidean field theories represents KO by relating field theories to self-adjoint trace-class semigroups of operators on a Hilbert space. In \cite{Cheung} field theories are identified with odd self-adjoint operators with compact resolvent defined on a graded subspace of a super Hilbert space.

In a lower dimension, \cite{HohnholdKreckStolzTeichner} shows that concordance classes of $0|1$-Euclidean (or topological) field theories yield de Rham cohomology, either with a $\mathbb{Z}/2$ grading or with the ordinary grading.

Combining the $1|1$ and $0|1$ dimensional pictures, it is possible to interpret the Chern character in terms of field theories and dimensional reduction. Along these lines see \cite{DumitrescuOnChern} and \cite{Han} for the ordinary Chern character, and \cite{Berwick-EvansHan}, \cite{Stoffel1}, and \cite{Stoffel2} for equivariant generalizations.

The larger picture is that one wishes to have geometric cocycles for the cohomology theory of topological modular forms. Segal \cite{Segalelliptic} suggested that $2$-dimensional conformal field theories would provide cocycles for some elliptic cohomology theory. These ideas were developed by Stolz and Teichner \cite{StolzTeichnerelliptic} who propose $2|1$-Euclidean field theories as the model for TMF. One result in this direction is the theorem that the partition function of a $2|1$-Euclidean field theory is a holomorphic modular function \cite{STsusy}. The $2|1$-dimensional case is an area of continuing work.

\subsection{Novelty}
Given the variety of constructions listed above, it is worth briefly noting what is distinctive about the new space of field theories presented here. The novelty is in the definition of field theory. Roughly speaking, we work with a more `natural' category of Euclidean bordisms, one in which the identity bordism is not separated from bordisms of positive length. This is spelled out in greater detail in Appendix~\ref{sec:bordismAppendix}, see in particular Figure~\ref{identityRelationFigure}. An importance consequence of such a bordism category is that the infinitesimal generator of the semigroup of endomorphisms defining the field theory must itself be a morphism in the algebraic target category.

To speak more precisely, we work with a presentation of such a bordism category, a presentation which we take for granted. Our definition of field theory (Definition~\ref{quadruples}) is in terms of generators and relations.

In order to accommodate the family of bordisms connecting the identity with intervals of positive length and obtain the correct homotopy type for the space of field theories it is necessary to work with strange families of vector spaces. Lemma~\ref{perturb} shows that spectra of operators work well in these strange families, but in order to prove this Lemma we need to restrict ourselves to families of vector spaces that are, in a suitable sense, finite. This finiteness is given in Definition~\ref{algebraicTarget}.

\subsection{Heuristics of the homotopy type}
Before delving into the substance of the paper we offer two sketches of a relation between $1$-dimensional geometric functorial field theories and K-theory.

\subsubsection{Supersymmetric quantum mechanics}
The notion of supersymmetric quantum mechanics in a spin manifold (\cite{Wittensusyqm}, \cite{Wittenindex}) yields a $1|1$-Euclidean field theory, and a bundle of spin manifolds determines a family of such field theories parameterized by the base. Given such a family one also has, by the Index Theorem for families, an element in the K-theory of the parameter space. We wish to relate field theories to K-theory in such a way that under this relation the family of field theories yields the index of the family.

\subsubsection{Group completion}
Oriented $1$-dimensional topological quantum field theories correspond, by evaluation on a point, to finite-dimensional vector spaces. The classifying space of all such field theories is then equivalent to $\coprod BGL_n$ and has a (homotopy) commutative monoid structure of direct sum. We would like to have another bordism category, with a forgetful functor to the oriented bordism category, in such a way that the induced map on classifying spaces (of categories of functors to vector spaces) is a group completion, and it turns out that the category of $1|1$-Euclidean bordisms is such a bordism category.

\subsection{Acknowledgements}
I thank Stephan Stolz and Augusto Stoffel for many helpful and enjoyable conversations.

\section{Definitions and Preliminary Matter}
\label{sec:Def}
\subsection{Conventions}

Here are some conventions of terminology and notation.

Given a supermanifold $S$, the underlying ordinary smooth manifold, known as the reduced manifold, will be denoted $S_{\mathrm{red}}$. The word `super' denotes $\mathbb{Z}/2$-graded. A contravariant functor from the category of smooth manifolds or supermanifolds to the category of sets will be called a generalized manifold or generalized supermanifold, respectively. We use $\cinf{S}$ to denote the global sections of the sheaf of functions on a supermanifold $S$. An excellent introduction to supermanifolds is \cite{DeligneMorgan}. 

Fix an infinite-dimensional separable super Hilbert space $H$ whose even and odd subspaces are each infinite-dimensional. Give $H$ a real structure as well, and pick an isomorphism $H \otimes H \to H$.

For set-theoretic reasons, we require that our supermanifolds be concrete in some way. Let a concrete supermanifold be a smooth manifold together with an odd finite rank real subbundle of the trivial bundle with fiber $H$, viewed as a supermanifold by taking the sheaf of algebras to be the sheaf of sections of the bundle of exterior algebras. By Batchelor's Theorem \cite{Batchelor} every supermanifold is realized in this way. We will use $\SMan$ to denote the category of concrete supermanifolds. Thus $\cinf{S}$ will consist of certain smooth $H$-valued functions on $S_{\mathrm{red}}$.

We denote by $\bbA^k$ the hyperplane in $\bbR^{k+1}$ consisting of points whose coordinates $(x_0, x_1, \ldots , x_k)$ sum to $1$. The embedding $|\Delta^k| \subset \bbA^k$ has the standard $k$-simplex as the subspace of $\bbA^k$ consisting of points all of whose coordinates are non-negative.

There is an subdivision endofunctor $sd$ of the category of simplicial sets. The subdivided $k$-simplex $sd \Delta^k$ has $k!$ non-degenerate $k$-simplices. More generally there are $(k!)^p$ non-degenerate $k$-simplices in the $p$-fold subdivision $sd^p \Delta^k$. We will use capitals $I$ to refer to an enumeration of these little simplices inside the repeatedly subdivided $k$-simplex. There are also affine linear barycentric subdivision maps $\bbA^k \to \bbA^k$ corresponding to these smaller simplices of the subdivided simplex, and they are denoted $b_I$. For a precise formula see Definition~\ref{barycentricMaps}.

\subsection{Algebraic Target Category}

A functorial quantum field theory is a symmetric monoidal functor from a category of bordisms to an algebraic target. Lacking a comprehensive description of our desired geometric bordism category, we defer discussion of it to Appendix~\ref{sec:bordismAppendix}. We will restrict ourselves at present to considering the images of generating objects and relations in the algebraic target. This target consists of submodules of sections of finite rank vector bundles. For the reader not concerned about the details of how we obtain sets rather than proper classes, the previous sentence suffices as a definition of an equivalent category.

\begin{defn}
\label{algebraicTarget}
  The stack $\target$ over concrete supermanifolds consists of
  \begin{itemize}
  \item Objects: Subsets of $\cinf{S_{\mathrm{red}}} \otimes H$, endowed with $\cinf{S}$-module structure compatible with the quotient map $\cinf{S} \to \cinf{S_{\mathrm{red}}}$, where $S$ is a concrete supermanifold. The modules are subject to the condition that they are submodules of sections of a finite rank super vector bundle which is a subbundle of $S_{\mathrm{red}} \times H$.
  \item Morphisms: A morphism from a $\cinf{S}$-module $M$ to a $\cinf{T}$-module $N$ is a a map $\cinf{T} \to \cinf{S}$ of super algebras and a grading preserving $\cinf{S}$-module map from $M$ to $\cinf{S} \otimes_{\cinf{T}} N$.
  \item Fibration: The map $\target \to \SMan$ is the forgetful functor.
  \end{itemize}  
\end{defn}
\noindent
We will use $\target_S$ to denote the fiber category over a supermanifold $S$. 

\begin{rmk}
  It could be tempting to impose finiteness by requiring that modules are simply finitely generated, though not necessarily projective. This will not work. The prototype of module we are interested in is smooth functions on the real line which are identically zero for $x \le 0$, viewed as a module over smooth functions on the real line.
\end{rmk}

\begin{rmk}
  We require that the supermanifolds be concrete so that modules always consist of $H$-valued smooth functions on the reduced manifold, even after pulling back along a map of supermanifolds. We then use the chosen isomorphism $H \otimes H \to H$ to view the pulled back module as a submodule of a bundle of sections. Concreteness is necessary so that we can have a simplicial set of field theories.
\end{rmk}

\subsection{Families of Field Theories}
\label{definingFamilies}

Before defining families of field theories it is necessary to define super semigroups. A family of field theories will roughly speaking be a super semigroup of endomorphisms of an object in $\target$. 

The supermanifold $\roo$ is a Lie group with the operation \[(s,\theta),(t,\eta) \mapsto (s + t + \theta \eta, \theta + \eta)\] using the functor of points formalism. We will be concerned with the semigroup $\roo_{>0}$ obtained by restricting the operation on $\roo$.

Generalized supermanifolds have a monoidal structure using Cartesian product of sets. We thereby speak of group or semigroup objects in generalized supermanifolds.

\begin{defn}
\label{superSemiGroup}
  A super semigroup in generalized supermanifolds is a semigroup $M$ in generalized supermanifolds together with a morphism $\roo_{>0} \to M$ of semigroups.
\end{defn}

Given a super vector space $Z$, we view it as a generalized supermanifold by the assignment $S \mapsto (\cinf{S} \otimes Z)^0$, where the superscript zero refers to taking the even part of the graded tensor product. We will be concerned with the case that $Z$ is the vector space of endomorphisms of a given super vector space $V$, in which case we speak of a super semigroup of endomorphisms of $V$.

We now define our field theories as certain quadruples associated with the category $\target$. We are thinking of restricting a symmetric monoidal functor to the generating objects and relations of the presentation in Theorem~\ref{presentationOfEbord}.

\begin{defn}
  \label{quadruples}
  A family of field theories $(V,W,L,R)$ parameterized by a smooth manifold $S$ is pair of objects $V$ and $W$ in $\target_S$, a morphism $L:W \underset{\cinf{S}}{\otimes} V \to \cinf{S}$ in $\target_S$, and an element $R \in \cinf{\roo_{>0}} \underset{{\bbR}}{\otimes} (V \underset{\cinf{S}}{\otimes} W)$, subject to the following relations.
  \begin{enumerate}
  \item The element $R$ determines a super semigroup of endomorphisms of $V$ or $W$ as it is composed with $L$ in the two possible ways. (See Figure~\ref{identityRelationFigure} for one such composition.)
  \item \label{identityCondition} The map $(s^2,\theta): \roo \setminus \{0\} \to \roo_{> 0}$, when used to pull back the endomorphisms of $V$ or $W$ determined by composing $R$ and $L$ in the two possible ways, extends at $0$ as the identity to a family of endomorphisms parameterized by $\roo$. Here $s$ is the standard coordinate function on $\bbR$. (Again, see Figure~\ref{identityRelationFigure}.) 
  \item The trace and supertrace of the endomorphisms of $V$ and $W$ determined by $L$ and $R$ give smooth functions on $\roo_{>0} \times S$. (This comes from the fact that field theories can be evaluated on closed bordisms, which here are circles.)
  \end{enumerate}
\end{defn}

A family of field theories parameterized by $S$ is to be thought of as an $S$-point of the mapping stack between a bordism category and the category $\target$. A morphism between two such families consists of maps over $S$ between the modules, compatible with the morphisms $R$ and $L$.

By means of the real structure on $H$, we can require that the modules $V$ and $W$ be real. This gives the space of real field theories. All the things we will prove go through equally well in the real and complex cases, and we do not distinguish them.

In general, when passing from a fibered functor $\mathcal{F} \to \mathcal{C}$ to a functor $\mathcal {C}^{op} \to \mathrm{Groupoids}$ one obtains only a weak (or pseudo) functor. In the present case, our concrete category means that we get a strict functor $\Man \to \mathrm{Groupoids}$, once we restrict to ordinary manifolds, since then we are just precomposing $H$-valued functions with smooth maps.

Thinking of field theories as a mapping stack or fibered category, the base site of the fibration is $\SMan$. The category $\Man$ of smooth manifolds is a subcategory, and there is an embedding $\Delta \hookrightarrow \Man$ sending $\Delta^k$ to affine $k$-dimensional space $\bbA^k$. In this way we obtain $\Delta^{op} \to \mathrm{Groupoids}$. Postcomposing with the functor forgetting morphisms yields a simplicial set. We make an adjustment to obtain compact simplices (see Definition~\ref{equivalentSimplices} and Definition~\ref{ssetFromGenMan}) in order to define simplicial sets $\susyeft_\bbC$ and $\susyeft_\bbR$. We refer to both, {\it mutatis mutandis}, as $\susyeft$. The purpose of the present paper is to determine the homotopy type of those simplicial sets.

\begin{rmk}\label{forgetMorphisms}
  The omission of natural transformations between field theories, so that $\susyeft$ is a simplicial set rather than a simplicial groupoid, is justified, just as it is appropriate to view the singular set of a topological space as a simplicial set rather than as a simplicial space.
\end{rmk}

\begin{example}
  In the K-theory of a point the zero vector space represents the same class as the sum of an even and odd line. This equivalence needs to be implemented in our space of field theories by a path of field theories between the topological theory assigning $\{0\}$ to a super point and the topological theory assigning an even plus an odd line to the super point. Here is a sketch of how such a path is constructed. A path in the space of field theories is a family of field theories parameterized by $\bbR$. Let $V$ be the module whose elements are smooth functions on $\bbR$ which are zero for $x \le 0$ and which for positive $x$ lie in some two-dimensional subspace of $H$ whose superdimension is zero. Let $W$ be the same module. Pick some suitable pairing $V \otimes W \to \cinf{\bbR}$ as $L$. Now define a super semigroup of endomorphsisms of $V$ by defining the infinitesimal generator as
\[\diracD =
  \begin{bmatrix}
    0 & \frac{1}{x} \\
    \frac{1}{x} & 0 \\
  \end{bmatrix}
\]
for positive $x$ and $0$ for non-positive $x$. (The infinitesimal generator is discussed further in Section~\ref{sec:ssg}.) Observe that $e^{-\diracD^2} $ has smooth trace and supertrace. Away from $x = 0$ we can modify $\diracD$ so that at $x=1$ it is the zero matrix. In that way we obtain a path of field theories between the topological theory with vector space $V = \{0\}$ (at $x = 0$) and the topological theory with vector space of dimension $2$ and superdimension $0$ at $x=1$. 

\end{example}

\begin{rmk}\label{restrictToPoints}
  Families of field theories can be pulled back along smooth maps. In the previous example, we can restrict along the inclusions of $0$ and $1$ in $\bbR$ to get individual field theories. Then the $\bbR$-family of field theories is a path between those two points in the space of field theories. We will regularly work by restricting a family to a point. Families of field theories are determined by their restrictions to each point.
\end{rmk}

\subsection{A classifying space}

We will determine the homotopy type of $\susyeft$ by establishing its equivalence with another space denoted $\nqvect$. Roughly speaking, $\nqvect$ is the nerve of a topological category.

Recall the infinite dimensional separable super Hilbert space $H$ whose odd and even parts are each infinite dimensional.

\begin{defn}
  The topological category $\mathrm{QVect}_\bullet$  consists of
  \begin{itemize}
  \item Objects: $\mathrm{QVect}_0$, the Grassmannian $Gr(H)$ of finite-dimensional super subspaces of the super Hilbert space $H$ endowed with its ordinary topology.
  \item Morphisms: $\mathrm{QVect}_1$, the space whose points are inclusions of super subspaces of $H$, $V_0 \subset V_1 \subset H$, together with an odd involution $\alpha$ on some complementary subspace $V_1'$ of $V_0$ in $V_1$. In other words we have $V_1 = V_0 + V_1'$, and an odd involution on the second summand. Via orthogonal projection onto a subspace we can view the involution as an element of $B(H)$. Give $\mathrm{QVect}_1$ a topology by considering it as a subspace of $Gr(H) \times Gr(H) \times B(H)$.
  \end{itemize}
\end{defn}

The homotopy type of a closely related space is given in \cite[Section 7]{HST}. Suitably modifying that result to our more general case of arbitrary complementary subspaces (rather than their orthogonal ones) leads to the following conclusion.

\begin{theorem} \label{htpyTypeQVect}
  The classifying space of the topological category $\mathrm{QVect}_\bullet$ represents real or complex K-theory, according to whether the Hilbert space and the subspaces are respectively real or complex.
\end{theorem}

The idea is that $\mathrm{QVect}_\bullet$ is Quillen's categorical group completion \cite{Grayson} of the category of finite dimensional inner product spaces with the monoidal structure of orthogonal direct sum.

There is also the analogous topological category for real super subspaces and inclusions. The homotopy type again comes from \cite[Section 7]{HST}. As in the case of field theories, we will not distinguish between the real and complex cases in our notation.

From $\mathrm{QVect}_\bullet$ we construct a bisimplicial set $\widehat{\nqvect}$ as follows. The $(k,l)$ bisimplices of $\widehat{\nqvect}$ are smooth maps

\begin{equation*}
  \bbA^k \to \overbrace{\mathrm{QVect}_1 \times_{Gr(H)} \mathrm{QVect}_1 \times_{Gr(H)} \ldots \times_{Gr(H)} \mathrm{QVect}_1}^l
\end{equation*}

\noindent
where the fibered product is with respect to the natural source and target maps.

A map $\bbA^k \to \mathrm{QVect}_1$ classifies a pair of nested vector bundles $V_0 \subset V_1$ on $\bbA^k$ and a continuous involution on a complement of $V_0$ in $V_1$. A collection of $l$ such maps is composable (i.e. in the nerve) when the larger bundle of one pair classified by $\bbA^k \to \mathrm{QVect}_1$ is the smaller bundle of the next pair. Thus, a $(k,l)$-bisimplex is a collection of $l+1$ bundles $V_0 \subset V_1 \subset \ldots \subset V_l$ on $\bbA^k$ together with involutions of complements of $V_i$ in $V_{i+1}$. The bisimplicial set $\widehat \nqvect$ is to be thought of as a smooth version of the level-wise singular set of the nerve of a topological category. 

\subsection{Generalized Manifolds and Simplicial Sets}
\label{gmanTosSetSection}
Given a generalized manifold $\Man^{op} \to Set$, restricting along the embedding $\Delta^{op} \to \Man^{op}$ yields a simplicial set. We modify the simplicial set resulting from such a process so that we have compact smooth simplices by introducing the following equivalence relation.

\begin{defn} \label{equivalentSimplices}
  Let $\mathcal{G}$ be a generalized manifold, and let $X$ and $Y$ be elements of $\mathcal{G}_{\bbA^k}$. Then $X$ and $Y$ are equivalent if for every smooth map $f:S \to \bbA^k$ whose image lies in $\Delta^k \subset \bbA^k$, the elements $f^*X$ and $f^*Y$ of $\mathcal{G}_S$ are equal.
\end{defn}

The equivalence relation above is used to define the simplicial set associated to a generalized manifold.

\begin{defn}
  \label{ssetFromGenMan}
  The simplicial set associated to a generalized manifold $\mathcal{G}$ is the restriction of $\mathcal{G}$ along $\Delta^{op} \to \Man^{op}$ modulo the relation of Definition~\ref{equivalentSimplices}.
\end{defn}

The definition above is how we speak of the simplicial set of field theories $\susyeft$. We do not distinguish between generalized manifolds and their associated simplicial sets in our notation. To be fully explicit we now spell out exactly what we mean by the simplicial sets $\susyeft$ and $\nqvect$.

\begin{defn}
  \label{susyeftDefn}
  The simplicial set $\susyeft$ has as $k$-simplices families of field theories parameterized by $\bbA^k$, modulo the relation of Definition~\ref{equivalentSimplices}. The face and degeneracy maps come from adding coordinates or inserting $0$.
\end{defn}

\begin{defn}
  \label{nqvectDefn}
  The simplicial set $\nqvect$ has as $k$-simplices maps
  \begin{equation*}
  \bbA^k \to \overbrace{\mathrm{QVect}_1 \times_{Gr(H)} \mathrm{QVect}_1 \times_{Gr(H)} \ldots \times_{Gr(H)} \mathrm{QVect}_1}^l
\end{equation*}
modulo the relation of Definition~\ref{equivalentSimplices} (modifying it slightly as necessary) with face and degeneracy maps arising from summing or repeating vector bundles.
\end{defn}

\section{Properties of Families of Field Theories}

In Definition~\ref{quadruples} we stated a definition of families of field theories in terms of data $(V,W,L,R)$ in $\target$. Now we clarify some of the consequences of the conditions imposed on those data. One is that super semigroups of operators have infinitesimal generators. Another is that the modules $V$ and $W$ in such a quadruple are dual to each other in a certain sense. A third property is that the families of operators arising from families of field theories have spectra that vary continuously.

\subsection{Infinitesimal Generator and Super Semigroup}
\label{sec:ssg}
In the case that $V$ is a finite-dimensional vector space (in which case we are considering a family of field theories parameterized by a $0$-manifold), super semigroups $\roo_{>0} \to \ENDO{V}$ correspond bijectively with odd endomorphisms of $V$ via

\begin{equation*}
\diracD \mapsto e^{-t \diracD^2}(1 + \theta \diracD)
\end{equation*}
\noindent

This statement admits the following generalization.

\begin{lemma} \label{generatorforbundles}
  Let $S$ be an ordinary smooth manifold. Suppose that an $S$-family of Euclidean field theories $E$ consisting of $(V,W,L,R)$ is such that the $\cinf{S}$-modules $V$ and $W$ are
  finitely-generated and projective (i.e. arise as sections of
  finite-rank vector bundles over $S$). Then the super semigroup associated to the family of field theories has an infinitesimal generator which is an odd bundle endomorphism of $V$.
\end{lemma}

\begin{proof}
  By restricting the family of endomorphisms parameterized by $\roo$ (whose existence is the condition (\ref{identityCondition}) of Definition~\ref{quadruples} a family of field theories must satisfy) along $\bbR^{0|1} \to \roo$, we obtain from a field theory a grading-preserving map
  \begin{equation*}
    V \to \cinf{\bbR^{0|1}} \otimes V
  \end{equation*}
which corresponds to a single odd endomorphism of $V$. At every point of $S$ the restriction of such an endomorphism determines the necessary super semigroup, and linearity of the endomorphism over $\cinf{S}$ means that we have a smooth bundle map. By exponentiation, this odd bundle endomorphism leads to the super semigroup of endomorphisms defining the family of field theories.
\end{proof}

Any odd bundle endomorphism leads by exponentiation to a super semigroup of endomorphisms.

If we are working with modules that do not arise from vector bundles, it is difficult to state exactly what sort of conditions exist on an infinitesimal generator. What is the case, though, is that a super semigroup of endomorphisms determines and is determined by a pair of maps $A: \bbR_{>0} \to \ENDO{V}^{0}$ and $B: \bbR_{> 0} \to \ENDO{V}^{1}$ subject to
  \begin{itemize}
  \item $A(s + t) = A(s)A(t)$
  \item $B(s+t)=A(s)B(t)$
  \item $A'(s+t) = B(s)B(t)$
  \end{itemize}

We will be concerned particularly with the ordinary semigroup $A(t)$ of endomorphisms. For related discussion see \cite[Lemma 3.2.14]{StolzTeichnerelliptic} and \cite[Proposition 5.9]{HST}.

\subsection{Pairing}

It will be necessary to know that in a $1|1$-Euclidean field theory the vector spaces $V$ and $W$ are isomorphic. Let $A(1)$ denote the endomorphism of $V$ given by composing $L$ and $R$ and evaluating the even semigroup at $1$. Let $\bar{A}(1)$ denote the corresponding endomorphism of $W$ given by composing $L$ and $R$ in the other possible way. Let $V_\lambda$ and $W_\lambda$ refer to generalized eigenspaces of $A(1)$ or $\bar{A}(1)$, respectively, with eigenvalue $\lambda$. 

\begin{lemma} \label{LPerfectPairing}
  The spectra of $A(1)$ and of $\bar{A}(1)$ coincide, and $L$ induces an isomorphism $W_\lambda \cong \dual{V_\lambda}$.
\end{lemma}
\begin{proof}
  For notational simplicity in the proof we let $A$ denote $A(1)$, and similarly for $\bar{A}$.

By associativity of composition (in the bordism category) we have, for any $v \in V$ and $w \in W$, the following.
\begin{equation*}
  L(\bar{A}w,v) = L(w,Av)
\end{equation*}

Suppose that $Av = \lambda v$ and $\bar{A} w = \mu w$. Then the equation
\begin{equation*}
\mu L(w,v) = L(\bar{A}w,v) = L(w,Av) = \lambda L(w,v)
\end{equation*}
implies that $L$ can pair eigenvectors non-trivially only if $\lambda = \mu$. One can then extend the argument inductively to generalized eigenvectors. Importantly, every vector must pair non-trivially with some vector, since $A(t)$ approaches the identity as $t$ approaches $0$.
\end{proof}

\subsection{Perturbation Lemma}

In a family of field theories the spectrum of the semigroup of endomorphisms varies continuously as a function of the parameter space.

We begin by clarifying what is meant by continuously varying spectrum when the dimension of the vector space (and hence the sum of the multiplicities of eigenvalues) is changing. Consider the set $\mathcal{C}$ of finite sets of points in $\bbC$ labelled with natural numbers. Define a topology by letting the basic open sets $U$ be determined by finite collections of disjoint open disks in $\bbC$. The disks are labelled with natural numbers unless the disk contains $0$. A point in $\mathcal{C}$ is contained in such a $U$ if all its points lie in the various open disks of $U$ and the labels of these points sum to the label of the disk (for disks not containing the origin). An endomorphism of a finite dimensional vector space determines a point in $\mathcal{C}$ by its eigenvalues labelled with multiplicity. In a continuous family of operators, in the sense of continuity used here, new parts of the spectrum can be created, but they must appear at the origin.

With the topology on such configurations in place, we have the following lemma which generalizes the statement that the spectrum of an endomorphism of a finite dimensional vector space depends continuously on the endomorphism.

\begin{lemma}
  \label{perturb}
  Let $V$ be an element of $\target$ parameterized by a smooth manifold $S$, and let $A : V \to \cinf{\bbR_{>0}} \otimes V$ be a smooth semigroup of endomorphisms which has smooth trace and extends as the identity at $t=0$ (where $t$ is the coordinate on $\bbR$). Then the spectrum of $A(1)$ varies continuously as a function of $S$.
\end{lemma}
\begin{proof}
  For the case of an ordinary vector bundle this is the statement that the spectrum of a matrix depends continuously on the entries. We must address the complication present in our case of non-projective modules (in which case an endomorphism might not locally be just a family of matrices). This discussion will also highlight the role of the fact that $A(0) = id_V$ (i.e. not restricting ourselves to a category of `positive' bordisms in the language of \cite{HST}).  Recall that we have required that the module $V$ is a submodule of some finite-rank vector bundle. As a consequence, there is some finite maximal possible dimension of $V_y$, the fiber of $V$ at $y$ (recall Remark~\ref{restrictToPoints}), for all points $y$ in the parameter space. Suppose that $m$ is the dimension of $V_x$, and that $n$ is the rank of the ambient vector bundle containing $V$.

  Denote the characteristic polynomial of $A(1)_x$ by $p(z)$. Consider the following collection of polynomials.
\begin{equation*}
  p(z), zp(z), z^2p(z),\ldots,z^{n-m}p(z)
\end{equation*}

Since $A(t)_x$ is invertible for all $t$, the smallest eigenvalue of $A(1)_x$ is some distance $\nu$ from $0 \in \bbC$. Now anytime we choose $\epsilon < \frac{\nu}{2}$, we can find a $\delta$ so that any polynomial with degree between $m$ and $n$ whose coefficients are within $\delta$ of those of $p(z), zp(z),\ldots,$ or $z^{n-m}p(z)$ will have roots that are within $\epsilon$ of the roots of $p(z)$ or $0$.

The trace of the operator $A(1)$ as well as its powers $A(1)^k$ are all smooth functions of the parameter space $S$. The Newton-Girard identities imply that matrices whose powers have similar traces have similar spectra. We can then find a neighborhood of $x$ such that the coefficients of the characteristic polynomial of $A(1)_y$ are within $\delta$ of $z^kp(z)$, where $k$ is defined as $k:= dim V_y - dim V_x$. We then conclude that the extra $k$ eigenvalues of $A(1)_y$ are all within a distance $\epsilon$ of $0 \in \bbC$. Recall that $\epsilon$ is less than half the norm of the least eigenvalue of $A(1)_x$. This means that in a neighborhood of $x$ we distinguish the eigenvalues coming from $A(1)_x$ and those which appear at the origin.
\end{proof}

\section{The map and its inverse}

The space $\nqvect$ has a known homotopy type, and we now compare it to $\susyeft$. In one direction we construct a map of simplicial sets. In the other direction we construct simplices after choosing some extra data.

\subsection{Families of field theories from inclusions of vector spaces}
\label{sec:defineMap}
Here we define a map from the known space $\nqvect$ to the space of field theories. The idea of the construction is this. On vertices, the map will send a super vector space to the topological field theory which assigns that vector space to the super point (topological meaning that all intervals are assigned the identity operator). To an edge of $\nqvect$, which consists of an inclusion of vector spaces having the same super dimension and an odd involution on a complementary subspace, we assign a path in the space of field theories. This path starts at one of the topological field theories and ends at the other one. On the way, it is passes through field theories that are not topological. The odd involution on the complementary subspace is used to define a non-zero infinitesimal generator of a super semigroup of endomorphisms.

Given a vector bundle $V$ on $\bbA^k$, let $V(x_0, \ldots ,x_k)$ denote its fiber at a point $(x_0, \ldots ,x_k)$. We are always working with bundles embedded in $H$, so such fibers are finite dimensional subspaces of $H$.
Fix a decreasing smooth function $\rho(x)$ on the positive reals which is equal to $\frac{1}{x^2}$ for $x$ near $0$ and is identically $0$ in a neighborhood of $1$.

\begin{defn}
\label{comparisonMap}
  The map $F: \nqvect \to \susyeft$ is defined in the following manner. Given a $k$-simplex
\[ V_0 \hookrightarrow V_1  \hookrightarrow V_2 \ldots  \hookrightarrow V_k \qquad (\sigma_1, \sigma_2, \ldots \sigma_k) \]
where $\sigma_i$ denotes an odd involution on a complement of $V_{i-1}$ in $V_i$, observe that we obtain a direct sum decomposition
\[V_k \cong V_0 \oplus V_1' \oplus V_2' \oplus \ldots \oplus V_k' \]
ny means of the odd involutions. This $k$-simplex is assigned to a family of field theories $(V,W,L,R)$ parameterized by $\bbA^k$ such that
\begin{itemize}
\item $V$ is the module of smooth functions $\bbA^k \to H$ with that property that 
\begin{equation*}
f(x_0,\ldots,x_k) \in V_k(x_0 \ldots, x_k)
\end{equation*}
\noindent
for all $(x_0 \ldots, x_k)$ and that satisfy
\begin{equation*}
  f(x_0, \ldots, x_k) \in V_i(x_0 \ldots, x_k)
\end{equation*}
\noindent
whenever $x_{i},\ldots,x_k = 0$. Furthermore, the projection of $f$ onto $V_j'$ is required to vanish to all orders at points $x_{i},\ldots,x_k = 0$ when $j \ge i$.
\item $W = V$
\item $L$ is the composition $W \otimes V \hookrightarrow H \otimes H \cong \dual{H} \otimes H \to \bbC$ (we fix for all time a $\bbC$-linear isomorphism $\dual{H} \cong H$)
\item $R$ is determined by a super semigroup of operators and the pairing $L$. The odd involutions $\sigma_i$ determine a direct sum decomposition as above, and with respect to this decomposition the odd endomorphism $\diracD$ is given as follows.
\[\diracD :=  0 \oplus \rho \left(\sum_{i=1}^k x_i \right) \alpha_1 \oplus \ldots \oplus \rho \left( \sum _ {i=j}^k x_i \right) \alpha_j \oplus \ldots \oplus \rho(x_k)\alpha_k \]
\end{itemize}
\end{defn}

The next lemma shows that it is acceptable for the spectrum of $\diracD$ to grow without bound near the boundary of the standard $k$-simplex, as it does in the definition above.

\begin{lemma} \label{continuousInfGen}
  Multiplication by $\frac{1}{x^2}$ is a continuous automorphism of the space of functions 
  \begin{equation*}
    V:= \{f \in \cinf{\bbR}|f(x) = 0 \textrm{ for } x \le 0\}
  \end{equation*}
 endowed with the topology induced by $V \subset \cinf{\bbR}$, where the latter space has its standard Fr\'{e}chet topology involving norms of derivatives on compact subsets of $\bbR$.
\end{lemma}
\begin{proof}
  The space of functions $V$ is a closed subspace of a Fr\'{e}chet space, and hence itself a Fr\'{e}chet space, and in particular complete and metrizable, so that the Open Mapping Theorem may be applied. Consequently we need only show that multiplication by $x^2$ is a continuous linear bijection, which is straightforward.
\end{proof}

\subsection{Principal Construction}

There is a fundamental construction which takes as input a $k$-simplex of field theories and gives as output a subdivided $k$-simplex of $\nqvect$. The construction satisfies properties so that roughly speaking one can think of it as an inverse to the map $F$ given in Definition~\ref{comparisonMap}. A number of choices will be made; the choices are homotopically insignificant. The idea of the construction is this. A family of field theories parameterized by a smooth manifold $S$ is locally a super vector bundle with an odd endomorphism together with a non-finitely generated non-projective sheaf of modules over the ring of functions (with an odd endomorphism). We ignore the unusual module sheaf and use the super vector bundle, along with the eigenspace decomposition coming from the odd endomorphism, to construct a filtered super vector bundle, which is the sort of thing one needs to construct simplices of $\nqvect$. Importantly, when summands of the unusual module vanish at a point they do so by going to $0$ in the spectrum of $A(1)$, by the continuity as shown in Lemma~\ref{perturb}.

Here is the precise definition of the data used to construct relevant simplices of $\nqvect$.

\begin{defn}
  A subdivision with resolvent data of a $k$-simplex of $\susyeft$ is the collection of data $\{\{(U_i,\mu_i)\},p,f,q\}$. The $U_i$ are open subsets of the standard $k$-simplex. Each $\mu_i$ is a positive real number less than $1$. The $p$ and $q$ are natural numbers, and $f$ is a set map from the collection of little simplices in a $p$-times subdivided $k$-simplex to the collection $\{i\}$ of indices for $U_i$ and $\mu_i$.  The data are subject to these conditions.
  \begin{enumerate}
  \item The $U_i$ constitute an open cover of the standard $k$-simplex.
  \item Each $\mu_i$ is such that at each point $x \in U_i$, the circle at the origin in $\bbC$ of radius $\mu_i$ does not intersect the spectrum of the endomorphism $A(1)_x$. In other words, all complex numbers of norm $\mu_i$ are in the resolvent of $A(1)_x$.
  \item A little simplex of the $p$-fold barycentric subdivision of the standard $k$-simplex lies entirely in at least one of the open sets $U_i$.
  \item The set map $f$ assigns to each little simplex an index $i$ in such a way that the little simplex lies within $U_i$. (In other words, if a little simplex lies in the intersection $U_i \cap U_j$ the map $f$ resolves the ambiguity of which index corresponds to this little simplex.)
  \item  The natural number $q \ge p$ is subject to the condition that for any index $I$ corresponding to $p$-fold subdivision, and $J$ any index corresponding to a $q$-fold subdivision, if $b_I(|\Delta^k|) \cap b_J(|\Delta^k|) \ne \varnothing$, then $b_J(|\Delta^k|) \subset U_{f(I)}$. This means that if some part of the boundary of one of the littlest simplices (after $q$ subdivisions) lies in one of the $U_i$ for a neighboring littlest simplex, then all of it must.   
  \end{enumerate}
\end{defn}

\begin{lemma}
  Every $k$-simplex of $\susyeft$ admits a subdivision with resolvent data.
\end{lemma}
\begin{proof}
  The existence of the finite open cover $U_i$ and the resolvent elements $\mu_i$ follow from Lemma~\ref{perturb} giving the continuity of the spectrum of $A(1)$, and the compactness of the standard $k$-simplex. Then the existence of a natural number $p$ satisfying the conditions follows from the Lebesgue Number Lemma. The map $f$ can be constructed freely; all that is necessary is that there is at least one such map. The existence of the natural number $q$ again follows from the Lebesgue Lemma.
\end{proof}

Because simplicial homotopy groups will be computed using the functorial fibrant replacement $Ex^\infty$ it is necessary to consider maps $sd^n \Delta^k \to \susyeft$ and not just individual $k$-simplices. This is not a problem, as recorded in the next lemma.

\begin{lemma}
When constructing subdivisions with resolvent data for the simplices determined by a map $sd^n \Delta^k \to \susyeft$, the extra data can be made compatible at the boundaries of simplices.   
\end{lemma}
\begin{proof}
  The spectrum of $A(1)$ is continuous. The maps $f$ can be constructed subject to constraints imposed by neighboring simplices. The subdivision parameters $p$ and $q$ can be increased without leading to any complication.
\end{proof}

Now we show what the above data give us. 

\begin{constr} \label{keyConstruction}
  From a $k$-simplex of $\susyeft$ and subdivision with resolvent data 
  \begin{equation*}
\{\{(U_i,\mu_i)\},p,f,q\}
\end{equation*}
 we construct a map of simplicial sets $sd^q \Delta^k \to \nqvect$.
\end{constr}

  Recall the notation $b_I$ and $b_J$, where $b_I$ denotes an inclusion $|\Delta^k| \to \bbA^k$ arising from $p$-fold barycentric subdivision, and $b_J$ a map associated with $q$-fold subdivision. The function $f$ assigns a resolvent radius $\mu_j$ to each index $I$ corresponding to a simplex of $sd^p \Delta^k$.

Let $J: \Delta^k \to sd^q \Delta^k$ be any such simplex. Recall that $k$-simplices of $\nqvect$ are given by inclusions of $k+1$ vector bundles and odd involutions on complementary subspaces. We define the sequence of vector bundles for $J$.  As $i$ ranges from $0$ to $k$ (seen as the vertices of $\Delta^k$), define a collection of positive real numbers $\lambda_i$ as follows.
\begin{equation*}
  \lambda_i := \underset{I}{\mathrm{max}}\{\mu_{f(I)}|J(i) \in \mathrm{im}(b_I) \}
\end{equation*}

\noindent
The maximum is taken over all indices $I$.

Observe that $i \mapsto \lambda_i$ is a decreasing function. This follows from the fact that for any simplex $\Delta^k\to sd \Delta^k$ of a subdivided simplex, as the vertices $i$ increase they are carried into the interior of increasingly higher dimensional faces of $sd \Delta^k$, and are thus incident with fewer neighboring simplices.

Denote by $_{\lambda_i}V$ the submodule of $V$ consisting of functions lying in the sum of eigenspaces of $A(1)$ whose eigenvalues have norm greater than $\lambda_i$. It is well-defined in some neighborhood of the compact set $b_J(|\Delta^k|)$. The module $b_J^*({_{\lambda_i}V})$ is a vector bundle over $|\Delta^k|$ which extends smoothly to a neighborhood of $|\Delta^k|$.

 A field theory supplies additional data useful for constructing the necessary involutions. The operator $A(1)$ gives an eigenspace decomposition $_{\lambda_{i+1}}V = {_{\lambda_{i+1}}V_{\lambda_{i}}} \oplus {_{\lambda_i}V}$. Here ${_{\lambda_{i+1}}V_{\lambda_{i}}}$ denotes the sum of eigenspaces whose eigenvalues have norm larger than $\lambda_{i+1}$ and smaller than $\lambda_i$. The infinitesimal generator of the family of field theories is diagonal with respect to this decomposition. On ${_{\lambda_{i+1}}V_{\lambda_{i}}}$ the infinitesimal generator $\diracD$ has no purely imaginary eigenvalues, since ${\lambda_{i}} < 1$ and $A(1) = e^{- \diracD^2}$. As a result, we can decompose ${_{\lambda_{i+1}}V_{\lambda_{i}}}$ further into
\begin{equation*}
  {_{\lambda_{i+1}}V_{\lambda_{i}}} = {_{\lambda_{i+1}}V_{\lambda_{i}}}^+ \oplus {_{\lambda_{i+1}}V_{\lambda_{i}}}^-
\end{equation*}
with the superscripts $+$ and $-$ denoting the sum of eigenspaces of the infinitesimal generator whose eigenvalues have positive (respectively negative) real part. They are of the same dimension since the infinitesimal generator is odd. With respect to this decomposition of ${_{\lambda_{i+1}}V_{\lambda_{i}}}$ define the odd involution $\alpha_{i+1} := Id \oplus -Id$.

We can now define a $k$-simplex of $\nqvect$ as the following sequence of vector bundles
\begin{equation*}
  b_J^*({_{\lambda_0}V}) \subset b_J^*({_{\lambda_1}V}) \subset \ldots \subset b_J^*({_{\lambda_k}V})
\end{equation*}
together with the involutions $ b_J^*\alpha_{i}$.
This concludes Construction~\ref{keyConstruction}.

\section{Equivalence}

We now show that the map $F$ defined in Section~\ref{sec:defineMap} is a homotopy equivalence. To do so we show that it induces isomorphisms on path components and all homotopy groups with all basepoints. Homotopy groups are determined simplicially after applying $Ex^\infty$ to obtain simplicial sets which are Kan.

In showing both injectivity and surjectivity we use Construction~\ref{keyConstruction}. The two parts of the proof have somewhat different flavors, though. Regarding injectivity --- since $\nqvect$ comes from a bisimplicial set there are some details to clarify about simplicial subdivisions. The key is Lemma~\ref{equivalenceOfSubdivisions}. A reader might wish to take for granted that the details about various subdivisions work out fine. On the surjectivity side --- again there is a bit to consider regarding subdivisions, but the more significant work is in the construction of a family of infinitesimal generators interpolating between an initial one and one lying in the image of $F$. This is accomplished in Equation (\ref{infgensubface}) in Proposition~\ref{surjectivity}.

We briefly recall a couple elements of simplicial homotopy. A $k$-simplex of $Ex^\infty X$ is a map $sd^n \Delta^k \to X$ with $n$ some natural number. Two such maps from subdivided simplices (possibly subdivided a different number of times) represent the same $k$-simplex of $Ex^\infty X$ if they are related by precomposition with $h: sd \Delta^k \to \Delta^k$, the last vertex map (see Definition~\ref{lastVertex}).

\subsection{Injectivity of induced map}
To prove that the map $F$ of Section~\ref{sec:defineMap} is injective on all homotopy groups for all basepoints, we apply Construction~\ref{keyConstruction} and take care of some combinatorial details.

Superdimension is preserved by $F$, and two vertices of $\susyeft$ are in the same component only if they have the same superdimension. Finally, two vector spaces of the same superdimension are connected (in $\nqvect$) by a zig-zag of their inclusions into their sum. Thus the map $F$ is injective on path components. Now we consider $\pi_k$ for $k \ge 1$.

\begin{prop}
  The induced map $F_* : \pi_k (\nqvect, E) \to \pi_k(\susyeft, E)$ is injective for all basepoints $E \in \nqvect_0$.
\end{prop}
\begin{proof}
  Given a representative
  \begin{equation*}
\gamma: sd^m(\Delta^k, \partial \Delta^k) \to (\nqvect, E)
\end{equation*}
of an element of the $k$-th homotopy group of $\nqvect$ suppose that there is an $n \ge m$ and a $\Gamma$ as below such that 
\begin{equation*}
  \Gamma: sd^n(\Delta^{k+1},\Lambda^{k+1}_{k+1}) \to (\susyeft,E)
\end{equation*}
has $(k+1)$st face equal to $F \circ \gamma \circ h^{n-m}$. The existence of such a $\Gamma$ is equivalent to $\gamma$ representing an element in the kernel of $F_*$. Then applying Construction~\ref{keyConstruction} to $\Gamma$ with some choice of subdivision with resolvent data yields a map $\hat{\Gamma}$ as below.
\begin{equation*}
  \hat{\Gamma}: sd^q(\Delta^{k+1},\Lambda^{k+1}_{k+1}) \to (\nqvect, E)
\end{equation*}

The $(k+1)$-face of $\hat{\Gamma}$ is, in the language of Lemma~\ref{equivalenceOfSubdivisions}, a partially degenerate barycentric subdivision of $\gamma$, and so there is a natural homotopy between the algebraic subdivision of $\gamma$ and this face. This natural homotopy together with $\hat{\Gamma}$ yield a homotopy relative boundary between $\gamma$ (seen as a map to $Ex^\infty \nqvect$) and the constant map at $E$.
\end{proof}

\subsection{Surjectivity of induced map}
Once again we use Construction~\ref{keyConstruction}, but this time we need to construct a family of field theories. This amounts to specifying a family of vector spaces together with an infinitesimal generator of a supersemigroup of endomorphisms of the family.

\begin{prop} \label{surjectivity}
  The induced map $F_* : \pi_k (\nqvect,E) \to \pi_k(\susyeft,E)$ is surjective for all basepoints $E$.
\end{prop}

\begin{proof}

  Applying Construction~\ref{keyConstruction} to a $\gamma$ representing an element of $\pi_k(\susyeft,E)$ (with some choice of subdivision with resolvent data) yields a map $\hat{\gamma}: sd^n\Delta^k \to \nqvect$ and then applying to this the map $F$, we obtain a subdivided simplex of $\susyeft$. We must produce a simplicial homotopy relative boundary connecting $\gamma$ and $F \circ \hat{\gamma}$, viewed as maps to $Ex^\infty \susyeft$. Without loss of generality we restrict attention to the case of beginning with a single $k$-simplex of $\susyeft$. All constructions will be sufficiently natural with respect to inclusions of faces.

The $k$-simplex of field theories involves a module $V$ and semigroup of operators with infinitesimal generator $\diracD$.
  
  We can pick a subdivision with resolvent data $\{\{(U_i,\mu_i)\},p,f,q\}$ for the given $k$-simplex. We then restrict our attention to one of the little $k$-simplices (after $q$-fold subdivision). On such a simplex we can decompose the module $V$ as 
\begin{equation*}
V = V_{\lambda_k} \bigoplus {_{\lambda_k}V_{\lambda_{k-1}}} \bigoplus {_{\lambda_{k-1}}V_{\lambda_{k-2}}} \bigoplus \ldots \bigoplus {_{\lambda_1}V_{\lambda_0}} \bigoplus {_{\lambda_0}V}
\end{equation*}
with the notation for the modules as before, meaning that ${_{\lambda_j}V_{\lambda_{j-1}}}$ is the subbundle of $V$ given by generalized eigenspaces of $A(1)$ corresponding to eigenvalues with norm between $\lambda_j$ and $\lambda_{j-1}$ (the $\lambda_i$ are defined by looking at the maximum of all relevant $\mu_j$).  Each of the summands is a vector bundle on a neighborhood of $|\Delta^k|$ except, possibly, for $V_{\lambda_k}$.

  Let $s$ be the homotopy/concordance parameter, i.e. the coordinate on $\bbA^1$. First, we define the module for the family of field theories. It consists of functions on $\bbA^k \times \bbA^1$ which lie in 
\begin{equation*}
{_{\lambda_k}V_{\lambda_{k-1}}} \bigoplus {_{\lambda_{k-1}}V_{\lambda_{k-2}}} \bigoplus \ldots \bigoplus {_{\lambda_1}V_{\lambda_0}} \bigoplus {_{\lambda_0}V}
\end{equation*}
\noindent
for $s \le 0$, and which lie in $V$ for $s > 0$. Observe that at $s = 0$ the non-vector bundle summand is missing.

We also require that for any function in $V$, the component which lies in $_{\lambda_i}V_{\lambda_{i-1}}$ vanishes to all orders at points $x_i=x_{i+1}=x_{i+2}=\ldots=x_k=s=0$.

Note that at $s=0$ the prescribed vanishing leads to a module that looks like something in the image of $F$.

We now define the family of endomorphisms on each summand separately.

On $_{\lambda_0}V$ define $A_s := A(s^4t)$ and $B_s := s^2B(s^4t)$. It is clear that these definitions give continuous endomorphisms. Moreover, they satisfy the super semigroup relations given in Section~\ref{sec:ssg} for all $s$. At $s = 0$ we get a topological theory.

On $V_{\lambda_k}$ define $A_s(t) := A(\frac{t}{s^4})$ and $B_s(t) := \frac{1}{s^2}B(\frac{t}{s^4})$. This means we speed up time using the homotopy parameter, and thereby wash out the non-vector bundle summand of $V$. This gives a continuous endomorphism, and satisfies the super semigroup relations. The only place we need to check continuity is at $s=0$, since otherwise we simply invoke the continuity of the original family $A(t)$. At any point in the parameter space $|e^{- \diracD ^2}| < \lambda_k < 1$ and so we see that the operator is near zero as $s$ gets small.

Finally, we treat $_{\lambda_{i}}V_{\lambda_{i-1}}$. On this summand, the difficulty in defining the endomorphisms is at the face consisting of those points for which $x_i = x_{i+1}=x_{i+2}=\ldots=x_k=0$. On this face, as $s$ approaches $0$, the infinitesimal generator of the semigroup must become infinite. Because this summand is a vector bundle in the module defining the original family of field theories, we may work directly with the infinitesimal generator of the semigroup, thanks to Lemma~\ref{generatorforbundles}.

Denote by $\diracD$ this infinitesimal generator. Our goal is to define a family of infinitesimal generators $\diracD_s$ with $\diracD_1 = \diracD$, and $\diracD_0$ being the generator which comes from the simplex in the image of $F$.

In order for the supertrace $str (e^{-t \diracD_s^2})$ to give a smooth function on the parameter space, it is necessary that $\diracD_s$ becomes infinite when summands of the module $V$ are `turned off.' Moreover, the operator must become infinite in a constrained way. It is not enough for the spectrum to move to $\infty$ in the Riemann sphere. Instead, because the spectrum of $A(1)$ varies continuously, the spectrum of $\diracD_s$ must grow large while remaining in the portion of the complex plane given by $|Im(z)| < |Re(z)|$; i.e. the portion consisting of complex numbers whose squares have positive real part.

In other words, we seek an operator $\diracD_s^{-1}$ which vanishes at $x_i = x_{i+1}=x_{i+2}=\ldots=x_k=s=0$ and which, near that subset, has a spectrum whose elements when squared have positive real part. 

We gather some useful items in order to define the desired operator. Define a new function denoted $\phi(s)$, a smooth step function satisfying $\phi(0)=1$ and $\phi(1) = 0$. Recall the operator $\alpha_i$, which arose in Lemma~\ref{keyConstruction}. This is the odd involution defined by the fact that the spectrum of $\diracD$ has no imaginary eigenvalues, since we are working only with $_{\lambda_{i}}V_{\lambda_{i-1}}$, on which $A(1) = e^{-\diracD^2}$ has eigenvalues of norm less than $\lambda_{i-1} < 1$. Specifically, $\alpha_i$ is defined as multiplication by $\mathrm{sign}(\mathrm{Re}(\mu))$ on generalized eigenvectors $v$ with eigenvalue $\mu $. Recall also the function of a single real variable $\rho(x)$, used in the definition of the map $F:\nqvect \to \susyeft$. This $\rho(x)$ is equal to $\frac{1}{x^2}$ for $x$ near $0$ and identically $0$ for $x$ near $1$. We use homogeneous coordinates $(x_0,x_1,\ldots,x_k)$ on $|\Delta^k|$ which sum to $1$.

Now we define $\diracD_s$ by defining its inverse and checking that it satisfies the spectral requirements enumerated earlier.

\begin{equation}
\label{infgensubface}
  \diracD_s^{-1} :=  s^2\diracD^{-1} + \phi(s)\frac{\alpha_i}{\rho(x_i + x_{i+1} + \ldots + x_k)}
\end{equation}

This operator vanishes exactly at $x_i = x_{i+1}=x_{i+2}=\ldots=x_k=s=0$. It certainly vanishes there, since both terms in the sum do. That it vanishes nowhere else can be seen in the following way. For any positive real numbers $a$ and $b$, the linear combination $a\diracD^{-1} + b \alpha_i$ is invertible. This is so since, if $\lambda$ is a complex number with positive real part, then $a\lambda + b$ is non-zero, and similarly if $\lambda$ has negative real part, then $a\lambda - b$ is non-zero. According to the description of the relation between $\alpha_i$ and the spectrum of $\diracD$, this proves the requisite invertibility. Finally, similar considerations show that $\diracD_s^{-1}$ has spectrum whose elements square to have positive real part.

The operator is defined for $s < 0$ as well as $x_i<0$, $x_{i+1}<0$, etc and is invertible there. This means that we can define the family in a neighborhood of the compact standard $k$-simplex
The expression (\ref{infgensubface}) clearly takes the correct form at $s=0,1$ in a neighborhood of $x_i = x_{i+1} = \ldots = x_k = 0$, so that it is a candidate for producing the necessary concordance between the initial family of field theories and the family in the image of $F$. The operator $\diracD_s$ is in fact a continuous endomorphism of the module $V$ on $|\Delta^k| \times \bbA$ defined above. The proof is as in Lemma~\ref{continuousInfGen}, where we apply the Open Mapping Theorem.

We now have a smooth infinitesimal generator defined on $|\Delta^k| \times \{0,1\} \cup \mathit{U}$, where $\mathit{U}$ is a neighborhood of the face on which the definition of Equation~\ref{infgensubface} restricts to the correct operators at $s=0,1$. This map can be smoothly extended to the whole of $|\Delta^k| \times |\Delta^1|$. Choosing a trivialization of the vector bundle, we identify a family of endomorphisms with a smooth map to $M_{n\times n}(\bbC)$. The partially-defined smooth infinitesimal generator then admits a smooth extension.

When defining the homotopy separately on each simplex of $sd^m\Delta^k$ we have not ensured that the families of infinitesimal generators are equal when restricted to shared faces. In order to construct a homotopy of maps $sd^m\Delta^k \to \susyeft$ we need to ensure that the homotopies we construct agree on these shared faces. This is always possible to do.

The homotopy constructed here is not precisely one between $(V,W,L,R)$ and an associated family arising from a simplex in $\nqvect$, since in passing to $\nqvect$ we forget about $W$. This is no complication, though. By Lemma~\ref{LPerfectPairing} we know that $L$ gives an isomorphism $W \cong \dual{V}$. An isomorphism between modules defining isomorphic families of field theories can be implemented by a smooth path of families of field theories. In the case that we are talking single field theories, i.e. families parameterized by a point, this is just the statement that the Grassmannian is path connected. 

That the homotopy constructed here is a homotopy relative boundary follows from the fact that the modification of the semigroups used to define the homotopy leaves field theories that are already topological unchanged. Since they have $A(t) = id$ and $B(t) = 0$ for all $t$, nothing happens to them during the reparameterization using $s$.

Lemma~\ref{equivalenceOfSubdivisions} gives a natural homotopy between the algebraic subdivision of $\gamma$ and the version of $\gamma$ involving barycentric subdivision involved in the homotopy just constructed.

This establishes the requisite homotopy relative boundary between $\gamma$ and $F \circ \hat{\gamma}$ as maps $(\Delta^k, \partial \Delta^k) \to (Ex^\infty \susyeft, E)$.
\end{proof}

\subsection{Conclusion}

Since the map $F:\nqvect \to \susyeft$ is an equivalence, and the homotopy type of the domain is known, we have determined the homotopy type of the space of field theories and proved Theorem~\ref{mainTheorem}. We remark briefly on changes necessary for the real case. By endowing $H$ with a real structure we can speak of real field theories and produce a space $\nqvect$ with homotopy type $BO \times \mathbb{Z}$. We can further require that $L$, $R$, and their composition are compatible with real structure present in $V$ and $W$. In this case, in Construction~\ref{keyConstruction} the summands $_{\lambda_{i+1}}V_{\lambda_i}$ in the eigenspace decomposition obtain real structures by restriction. The odd involutions are defined by the sign of the real part of the eigenvalue, which is unchanged by the action of the real structure, and thus the odd involution is also an involution on the real bundle by restriction.

\appendix

\section{Simplicial Sets}

For simplicial homotopy see \cite{GoerssJardine}. The simplicial set $\Delta^k$ is the nerve of the partially ordered set of natural numbers from $0$ to $k$ ordered by `greater than.' The simplicial set $sd \Delta^k$ is the nerve of the partially ordered set of subsets of $\{0,1,\ldots,k\}$, ordered by inclusion. By thinking of these simplicial sets in this way we define the last vertex map, necessary for treating endofunctors subdivision $sd$ and its right adjoint $Ex$ of the category of simplicial sets.

\begin{defn} \label{lastVertex}
  The last vertex map $h: sd \Delta^k \to \Delta^k$ sends a vertex $\{v_0,v_1,\ldots, v_j\}$ to its largest element.
\end{defn}

The $k$-simplices of $Ex X$ correspond to simplicial set maps from $sd \Delta^k$ to $X$. Applying $Ex$ repeatedly to an arbitrary simplicial set yields a simplicial set which is Kan and equivalent to the original. Hence $Ex^\infty$ is a functorial fibrant replacement.

We use $\bbA^k$ to denote the subset of $\bbR^{k+1}$ consisting of $(k+1)$-tuples $(x_0,x_1,\ldots,x_k)$ satisfying $\sum x_i = 1$.

The embedding $\Delta \to \Man$ is by $[k] \mapsto \bbA^k$, with coface and codegeneracy maps given by insertion of $0$ and summation of adjacent coordinates, respectively.

The non-degenerate $k$-simplices of $sd \Delta^k$ can be enumerated by bijections $\sigma:\{0,1,\ldots,k\} \to \{0,1,\ldots,k\}$ in the following way. The non-degenerate $k$-simplices of $sd \Delta^k$ are determined by increasing sequences of subsets $\{m_0\} \subset \{m_0,m_1\} \subset \ldots \subset \{m_0,m_1,\ldots,m_k\}$. Assign to such a sequence the permutation $\sigma: i \mapsto m_i$. This establishes a bijection between the non-degenerate $k$-simplices of $sd \Delta^k$ and elements of the group of permutations on $k+1$ letters.

Let $v_i$ be the point in $\bbR^{k+1}$ whose $i$-th coordinate is $1$ and all others are zero.

\begin{defn}
  \label{barycentricMaps} 
  Having fixed a permutation $\sigma \in S_{k+1}$, define an affine linear barycentric subdivision map $b_\sigma:\bbA^k \to \bbA^k$ by the following formula.
\begin{equation*}
  b_\sigma: v_i \mapsto \frac{1}{i+1}\left(v_{\sigma(0)} + v_{\sigma(1)} + v_{\sigma(2)} + \ldots + v_{\sigma(i)}\right)  
\end{equation*}
  
\end{defn}

A simplicial generalized manifold defines a simplicial set by precomposition with $\Delta \to \Delta \times \Delta \to \Delta \times \Man$, with the first map the diagonal and the second the product of the identity and the embedding $[k] \mapsto \bbA^k$.  Let $X$ denote both a simplicial generalized manifold and its corresponding simplicial set. 

The vertices of $sd \Delta^k$ are naturally identified with subsets $\{0,1,\ldots,k\}$. Call a map $f: sd^q \Delta^k \to \Delta^k$ face-preserving if the image under $f$ of each vertex $v$ of $sd^q \Delta^k$ is contained in the face corresponding to the subset $S_v \subset \{0,1,\ldots,k\}$ to which $v$ is taken by the $(q-1)$-fold application of the last vertex map. The point of defining face-preserving maps is that the face-preservation condition will ensure that a certain simplicial homotopy can be constructed in a specified direction.

A $k$-simplex $x$ of $X$ is an element of $X_{([k],\bbA^k)}$. Let $f: sd^q \Delta^k \to \Delta^k$ be any face-preserving map.

\begin{defn} \label{partiallyDegenDefn}
The partially degenerate barycentric subdivision of $x$ associated to $f$ is given by pulling $x$ back along the product of $f$ with the iterated barycentric subdivision maps.
\end{defn}

This partially degenerate barycentric subdivision is a map of simplicial sets $sd^q \Delta^k \to X$. In other words, it consists of $((k+1)!)^q$ elements of $X_{([k],\bbA^k)}$ whose various faces are equal in the necessary way. Observe that this subdivision is different than the purely algebraic subdivision defined in terms of precomposition with the last vertex map after taking the diagonal. Hence we need to relate the two.

\begin{lemma} \label{equivalenceOfSubdivisions}
  Let $f: sd^q \Delta^k \to \Delta^k$ be any face-preserving map, and $x:\Delta^k \to X$ be a $k$-simplex of the diagonal of the (restriction of a) simplicial generalized manifold. There is a natural simplicial homotopy from the partially degenerate barycentric subdivision associated to $f$ to the algebraic subdivision of $x$, i.e. the restriction of the homotopy to any face depends only on the restriction of $f$ to that face.
\end{lemma}

\begin{proof}
Fix a non-degenerate $k$-simplex $y$ of $sd^q \Delta^k$. The iterated barycentric subdivision, restricted to this simplex, corresponds to (precomposition with) one map $\bbA^k \to \bbA^k$ and the algebraic subdivision corresponds to another map $\bbA^k \to \bbA^k$. Interpolate linearly between these two to obtain a map $H:\bbA^{k+1} \cong \bbA^k \times \bbA^1 \to \bbA^k$.

There are $k+1$ non-degenerate $(k+1)$-simplices of $\Delta^k \times \Delta^1$. For $j$ ranging from $1$ to $k+1$, define maps $g_j: \bbA^{k+1} \to \bbA^{k+1} \cong \bbA^{k} \times \bbA^1$ by 

\begin{equation} 
 g_j(v_i) =
 \begin{cases}
    (v_i,v_1) & i \le j\\
     (v_{i-1},v_2)& i > j
  \end{cases}
\end{equation}

In the expression above, $v_i$ denotes the vector $(\delta_i^j)$ in $\mathbb{R}^{k+2}$, as it did earlier.

Compose $g_j$ and $H$ to pull back elements of $X_{([k],\bbA^k)}$ to elements of $X_{([k],\bbA^{k+1})}$. These elements are now modified by the face-preserving map $f$. For each $j$ produce a map $[k+1] \to [k]$ as follows. The given $f$, restricted to $y$, corresponds to some $f_{|y}:[k] \to [k]$, and similarly the iterated composition $h^q$ of the last vertex map yields a $h^q_{|y} : [k] \to [k]$. Produce a map $[k+1] \to [k]$ by the following.

\begin{equation}
i \mapsto
  \begin{cases}
    f_{|y}(i),&i \le j\\
    h^q_{|y}(i-1)& i > j
  \end{cases}
\end{equation}

Here, the fact that $f$ is face-preserving ensures that the map thus defined is order preserving. Pulling back along these maps gives elements of $X_{([k+1],\bbA^{k+1})}$, which is to say $(k+1)$-simplices of $X$, and by construction they together constitute a homotopy from the barycentric subdivision (restricted to $y$) to the restricted algebraic subdivision (restricted to $y$). Performing the construction on each non-degenerate $k$-simplex of $sd^q \Delta^k$ gives the homotopy. 
\end{proof}

\section{The Euclidean Bordism Category} \label{sec:bordismAppendix}

This appendix sketches a category of oriented, 1-dimensional, supersymmetric  Euclidean bordisms.

We begin with Euclidean isometries. Let $\AUT{\bbR^{1|1}}$ be the functor which assigns to each supermanifold $S$ the set of invertible elements in $\SMan(S \times \bbR^{1|1},\bbR^{1|1})$. Invertible here means that there exists another such element so that the composition of maps over $S$ gives the identity on $S \times \bbR^{1|1}$. We define a subgroup of $\AUT{\bbR^{1|1}}$. This will be the Euclidean group of isometries of $\bbR^{1|1}$. Note that given an $S$-point of $\bbR^{1|1}$, we can obtain an invertible element in $\AUT{\bbR^{1|1}}$ by translation as follows. 
\[
\mu \circ (f \circ p_1 \times id_{\roo} \circ p_2) : S \times \roo \to \roo \times \roo \to \roo
\]
\noindent
Then the $S$-point $(s,\theta)$ is invertible with inverse $(-s, -\theta)$.

There are also reflections. The assignment $(s,\theta) \mapsto (s, -\theta)$ gives an action of $\mathbb{Z}/2$ on $\roo$. This action in fact gives an automorphism of $\roo$ as a super Lie group. We define the Euclidean isometries $\mathrm{Iso}(\roo)$ of $\roo$ to be this semi-direct product $\roo \rtimes \mathbb{Z}/2$.

Having established the relevant isometry group we may now define the objects and morphisms of the bordism category $\susyebord$. The category we define will be closely modeled on the bordism category of \cite[Section 6]{HST}. Two important differences between this bordism category and that of \cite{HST} are that we add orientations and that we do not require that the reduced part of a family of bordisms forms a fiber bundle over the parameter space. This latter condition is what is called `positive' in \cite{HST}.

First, we define what it means to have a family of oriented Euclidean supermanifolds.

\begin{defn}
  A family of oriented Euclidean $1|1$-manifolds over a supermanifold $S$ is
  \begin{itemize}
  \item A map of supermanifolds $p:Y \to S$
  \item An atlas of open sets ${\mathit{U}_i}$ on $Y_{\mathrm{red}}$ and isomorphisms 
    \begin{equation*}
      Y_{|\mathit{U}_i} \cong \mathit{V}_i \subset S \times \bbR^{1|1}
    \end{equation*} 
    compatible with the projection map, such that the transition maps are given by maps  $p_1(\mathit{V}_{ij}) \to \mathrm{Iso}(\bbR^{1|1})$
  \item An orientation on the reduced fibers of $p:Y \to S$
  \end{itemize}
\end{defn}
\noindent

A family of Euclidean $0|1$-manifolds is defined similarly, but with $\mathrm{Iso}(\bbR^{0|1}) \cong \mathbb{Z}/2$.

Now we define an object of $\susyebord$ over a supermanifold $S$.

\begin{defn}
  An object of $\susyebord$ over the supermanifold $S$ is
  \begin{itemize}
  \item A family of oriented $1|1$-Euclidean supermanifolds $Y$ over $S$
  \item A family of $0|1$-Euclidean supermanifolds $Y^c$ over $S$, referred to as the core
  \item An isometric embedding $Y^c \to Y$ over $S$
  \item A decomposition $Y^\pm$ of $Y_{\mathrm{red}} \setminus Y_{\mathrm{red}}^c$ into the union of two open subsets, both containing $Y^c_{\mathrm{red}}$ in their closure
  \end{itemize}
  
  These data must satisfy the condition that the map of reduced manifolds from $Y^c_{\mathrm{red}} $ to  $S_{\mathrm{red}}$ is proper.
\end{defn}

A morphism of $\susyebord$ over a supermanifold $S$ is defined as an equivalence class. Representatives have the following form.

\begin{defn} \label{morphismRepresentatives}
  A morphism of $\susyebord$ over the supermanifold $S$ is represented by
  \begin{itemize}
  \item A family of oriented $1|1$-Euclidean supermanifolds $\Sigma$ over $S$
  \item A pair of objects $Y_0$ and $Y_1$ over $S$ (the source and target, respectively)
  \item Neighborhoods $W_0$ and $W_1$ of the cores $Y^c_0$ and $Y_1^c$
  \item Isometric embeddings $W_0 \to \Sigma$ and $W_1 \to \Sigma$ over $S$
  \end{itemize}
  The data must satisfy three conditions. First, the core of $\Sigma$, denoted $\Sigma^c$ and defined as $\Sigma_{\mathrm{red}} \setminus \{W_{0,red}^+ \bigcup W_{1,red}^-\}$ is such that the map $\Sigma^c \to S_{\mathrm{red}}$ is proper. Here $W_{i,red}^\pm$ refers to the decomposition of $W_{i,red} \setminus Y^c_{i,red}$ induced by the decomposition $Y_i^\pm$  associated with $Y_i$. Second, the image of the reduced part of the map $W^+_1 \to \Sigma$ is contained in the core $\Sigma^c$. Third, the maps $W_i \to \Sigma$ preserve orientations.
\end{defn}

\begin{defn}
  Two representatives $\Sigma_1$ and $\Sigma_2$ of morphisms as in Definition~\ref{morphismRepresentatives} represent the same morphism in $\susyebord$ when there are open neighborhoods $U_i$ of the cores $\Sigma_i^c$, restrictions of the neighborhoods $W_i$, and an orientation preserving isometric isomorphism $U_0 \to U_1$ over $S$ which induces isomorphisms on the restrictions of the $W_i$.
\end{defn}

 We can use the super Lie group structure of $\roo$ to define a basic example of a family of bordisms parameterized by the supermanifold $\roo_{>0}$. We refer to this as the universal family of super-intervals. Note that $\roo_{>0}$ denotes the restriction of $\roo$ to the open subset of $\bbR$ consisting of the positive real numbers. 

The (positively oriented) super point, which we denote by $spt$, is the object of $\susyebord$ given by $\roo$ with the core inclusion $\bbR^{0|1} \hookrightarrow \roo$ and the decomposition of $\bbR \setminus \{0\}$ into the positive and negative axes, and the standard orientation.

\begin{defn} \label{universalInterval}
The universal family of intervals is constructed in the following manner. Pull $spt$ back to $\roo_{>0}$ by the unique map $\roo_{>0} \to pt$, and continue to refer to this object as $spt$. The universal family of intervals, a bordism over $\roo_{>0}$ which is an endomorphism of $spt$, is defined as $\roo_{>0} \times \roo$ with standard orientation. The inclusion of the source is the identity, and the inclusion of the target is given by 
  \begin{equation*}
    (s,\theta) (t, \eta) \mapsto (s,\theta) (s + t + \theta \eta, \theta + \eta)
  \end{equation*}
\noindent
In words, we use the super Lie group structure of $\roo$ to translate the super point when including it as the codomain.
\end{defn}

The negatively oriented super point is like the positive one, except for the choice of the other orientation of the real line. We denote the negatively oriented super point by $\overline{spt}$.

Having defined the objects and morphisms of the category, we now combine them to define $\susyebord$.

\begin{defn}
  The symmetric monoidal fibered category $\susyebord$ consists of the following. The objects of $\susyebord$ are families of Euclidean $0|1$-manifolds. The morphisms are equivalence classes of families of Euclidean $1|1$-manifolds together with inclusions of source and target. Composition of bordisms is by gluing.  The fibration $\susyebord \to \SMan$ maps a family to its parameter space. The symmetric monoidal structure is by taking disjoint unions of families of supermanifolds.
\end{defn}

Multiplication by $-1$ on odd coordinates yields an automorphism of the identity on $\susyebord$, and we refer to such a natural transformation as a flip. 

\begin{theorem}  \label{presentationOfEbord}  
The category $\susyebord$ admits the following presentation as a symmetric monoidal fibered category with flip. The following objects and bordisms are the generators.

\begin{itemize}
\item $spt$ and $\overline{spt}$ the positively and negatively oriented superpoints, which are families over $pt \in \SMan$
\item $\mathcal{L}:\overline{spt} \coprod spt \to \varnothing$ a bordism over $pt$
\item $\mathcal{R}: \varnothing \to spt \coprod \overline{spt}$ a family of bordisms over $\roo_{>0}$, with the inclusion of one of the points defined by translation using the group structure on $\roo$ (as in the universal family of intervals discussed in~\ref{universalInterval})
\end{itemize}

These data are subject to two relations, spelled out in detail below. One relation expresses that the family $\mathcal{R}$ composes with itself along $\mathcal{L}$ in a way determined by the Lie group structure of $\roo$. The other relation is that the identity is a limit of the family of endomorphisms given by $\mathcal{R} \circ \mathcal{L}$.
\end{theorem}

See \cite[Section 6]{HST} for a further discussion of the presentation of a closely related bordism category as well as a sketch of a proof.

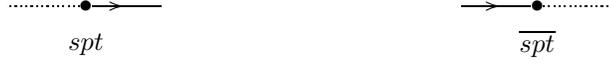
\begin{figure}
  \centering
  \qquad
  \begin{xy}
    (-40,0)*{}="POSL"; (-30,0)*{\bullet}="POSM";(-20,0)*{}="POSR";(20,0)*{}="NEGL";(30,0)*{\bullet}="NEGM";(40,0)*{}="NEGR";
    "POSL";"POSM"**\dir{.};
    "POSM";"POSR"**\dir{-};?(.5)*\dir{>};      
    "NEGL";"NEGM"**\dir{-};?(.5)*\dir{>};
    "NEGM";"NEGR"**\dir{.};
    (-30,-5)*{spt};(30,-5)*{\overline{spt}};
  \end{xy}
  \caption{Generating objects of the bordism category}
  \label{generatingObjects}
\end{figure}

Figure~\ref{generatingObjects} depicts the reduced manifold part of the generating objects. In that figure, the solid portion of the interval is the part $Y^+$ that enables gluing of bordisms along shared objects and which is involved in one of the conditions for bordisms; namely, it is the part that must be embedded in the core of the bordism when the object is part of the codomain. The arrow denotes the orientation on the ambient $1$-manifold.

Figure~\ref{generatingMorphisms} depicts the reduced manifold part of the generating bordisms. In the figure the source of a bordism is on the right, and the target on the left, to relate gluing to the way that composition of functions is usually written. The inclusions of the objects are implied by the arrow which indicates the orientation. We use the convention that $\mathcal{L}$ is a morphism from $\overline{spt} \coprod spt$ to the empty set, and that $\mathcal{R}$ is a morphism from the empty set to $spt \coprod \overline{spt}$.

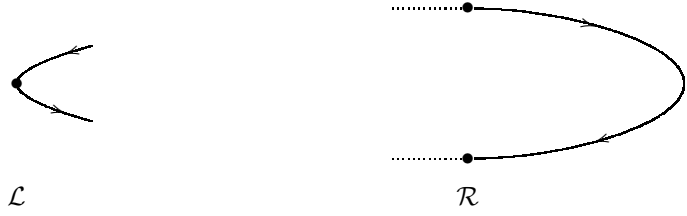
\begin{figure}
  \centering
  \qquad
  \begin{xy}
    (-40,-5)="LTOP";(-40,-15)="LBOT"; (-50,-10)*{\bullet}="LMID";
    "LTOP";"LBOT" **\crv{"LTOP" & (-60,-10) & "LBOT"};?(.2)*\dir{>};?(.8)*\dir{>};
    (0,0)*{}="A"; (0,-20)*{}="B"; (10,0)*{\bullet}="TOP";
    (10,-20)*{\bullet}="BOTTOM"; "TOP";"BOTTOM" **\crv{(20,0) &
      (40,-5) & (40, -15) & (20,-20) };?(.2)*\dir{>};?(.8)*\dir{>};
    "A";"TOP"**\dir{.}; "B";"BOTTOM"**\dir{.};
    (-50,-25)*{\mathcal{L}};(10,-25)*{\mathcal{R}};
  \end{xy}
  \caption{Generating morphisms of the bordism category}
  \label{generatingMorphisms}
\end{figure}

We now describe the two relations. 

The first is a super semigroup relation. Given maps $f,g:S \to \roo$, one can construct a family of bordisms over $S$ by pulling back $\mathcal{R}$ along $f$ and $g$ and then composing them via $\mathcal{L}$. One can also use the super Lie group structure of $\roo$ so that $f$ and $g$ determine a single map

\begin{equation*}
  \mu \circ (f \times g) : S \to \roo \times \roo \to \roo 
\end{equation*}
\noindent
to $\roo$ and then pull back the family $\mathcal{R}$ along this map. The relation is that these two constructions produce isomorphic families of bordisms over $S$. The relation is depicted in Figure~\ref{semigroupRelationFigure}, with the vertically stacked bordisms on the right glued along $\mathcal{L}$ to produce the bordism on the left.

\begin{figure}
  \centering
 \begin{xy}
    (0,0)*{}="A"; (0,-30)*{}="B"; (10,0)*{\bullet}="TOP";
    (10,-30)*{\bullet}="BOTTOM"; "TOP";"BOTTOM" **\crv{(30,0) &
      (40,-7) & (40, -22) & (30,-30) };?(.2)*\dir{>};?(.8)*\dir{>};
    "A";"TOP"**\dir{.}; "B";"BOTTOM"**\dir{.};
    (60,-15)*{=};
    (80,0)*{}="A"; (80,-20)*{}="B"; (90,0)*{\bullet}="TOP";
    (90,-20)*{\bullet}="BOTTOM"; "TOP";"BOTTOM" **\crv{(100,0) &
      (120,-5) & (120, -15) & (100,-20) };?(.2)*\dir{>};?(.8)*\dir{>};
    "A";"TOP"**\dir{.}; "B";"BOTTOM"**\dir{.};
    (100,-20)="LTOP";(100,-30)="LBOT"; (90,-25)*{\bullet}="LMID";
    "LTOP";"LBOT" **\crv{"LTOP" & (80,-25) & "LBOT"};?(.2)*\dir{>};?(.8)*\dir{>};
    (80,-30)*{}="AA"; (80,-40)*{}="BA"; (90,-30)*{\bullet}="TOPA";
    (90,-40)*{\bullet}="BOTTOMA"; "TOPA";"BOTTOMA" **\crv{(100,-30) &
      (110,-32) & (110, -37) & (100,-40) };?(.2)*\dir{>};?(.8)*\dir{>};
    "AA";"TOPA"**\dir{.}; "BA";"BOTTOMA"**\dir{.};
     (25,-50)*{\mathcal{R}_{(s + t+ \theta\eta,\theta + \eta)}};(90,-50)*{\mathcal{R}_{(s,\theta)} \underset{\mathcal{L}}{\circ} \mathcal{R}_{(t,\eta)}};
   \end{xy}  
\caption{Semigroup relation}
  \label{semigroupRelationFigure}
\end{figure}
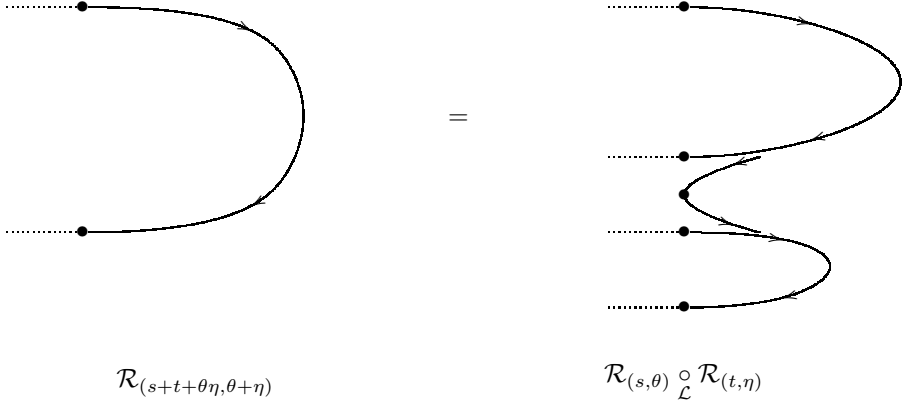

The second relation is that there is a family connecting the identity with the intervals of positive length. The map $(s^2,\theta): \roo \setminus \{0\} \to \roo_{> 0}$ defines a family of bordisms by pulling back $\mathcal{R}$ to the complement of $0 \in \roo$. Composing with $\mathcal{L}$, we obtain endomorphisms of $spt$ or $\overline{spt}$. The family thus defined (which has parameter space $\roo \setminus \{0\}$) extends as the identity at $0 \in \roo$. The relation is depicted in Figure~\ref{identityRelationFigure}. In that figure, we show that the composition of $\mathcal{R}$ and $\mathcal{L}$ on the right is the restriction of a family connected to the identity, namely the family on the far left, in which the vertical axis depicts the reduced length of the interval $\mathcal{I}_{(s^2,\theta)}$.
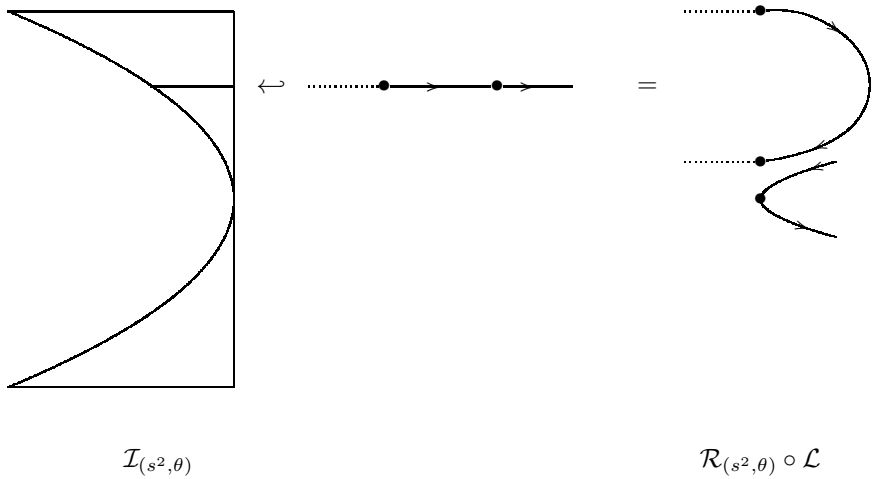
\begin{figure}
  \centering
  \begin{xy}
    (85,-10)*{=};
    (90,0)*{}="A"; (90,-20)*{}="B"; (100,0)*{\bullet}="TOP";
    (100,-20)*{\bullet}="BOTTOM";"TOP";"BOTTOM" **\crv{(100,0) &
      (105,1) & (115,-5)  & (115,-15)  & (107, -19)};?(.4)*\dir{>};?(.9)*\dir{>};
    (110,-20)="LTOP";(110,-30)="LBOT"; (100,-25)*{\bullet}="LMID";
    "LTOP";"LBOT" **\crv{"LTOP" & (90,-25) & "LBOT"};?(.2)*\dir{>};?(.8)*\dir{>};
    (90,0)*{}="X";(90,-20)*{}="Y";
    "TOP";"X"**\dir{.};
    "BOTTOM";"Y"**\dir{.};
    (50,-10)*{\bullet}="IL";
    (65,-10)*{\bullet}="IR";
    (40,-10)="ILL";
    (75,-10)="IRR";
    "ILL";"IL"**\dir{.};"IL";"IR"**\dir{-};?(.5)*\dir{>};
    "IR";"IRR"**\dir{-};?(.5)*\dir{>};
    (0,0)="PUL";
    (30,0)="PUR";
    (0,-50)="PLL";
    (30,-50)="PLR";
    (60,-25)="PMID";
    "PUL";"PLL";**\crv{"PMID"};
    "PUL";"PUR"**\dir{-};
    "PUR";"PLR"**\dir{-};
    "PLL";"PLR"**\dir{-};
    (19,-10)="JL";
    (30,-10)="JR";
    "JL";"JR";**\dir{-};
    (35,-10)*{\hookleftarrow};
    (100,-60)*{\mathcal{R}_{(s^2,\theta)}\circ \mathcal{L}};    
    (20,-60)*{\mathcal{I}_{(s^2,\theta)}};
   \end{xy} 

\caption{Identity relation}
  \label{identityRelationFigure}
\end{figure}

\begin{rmk}
  The second relation constrains the type of linear map that
  can arise as the image of a super interval under a field theory; the
  semigroup property of the previous relation together with this
  family imply that such a linear map cannot have a kernel.
\end{rmk}

General families of $1|1$-Euclidean manifolds over a parameter space $S$ are neither connected nor are they necessarily pulled back along maps $S \to \roo_{> 0}$. Nonetheless, symmetric monoidal functors out of the bordism category with codomain a stack are essentially defined on the generating objects and bordisms above, since general families are assembled from those, using sufficiently refined covers and the monoidal structure. One also imposes the spin-statistics condition (see \cite[Definition 6.44]{HST}) that the flip on the bordism category be sent to the grading involution on the corresponding algebraic objects.

\end{document}